\documentclass{amsart}
\usepackage{mathrsfs}
\usepackage{amssymb}
\usepackage{xy}

\let\sma\wedge
\newcommand{\htp}{\simeq}
\renewcommand{\to}{\mathchoice{\longrightarrow}{\rightarrow}{\rightarrow}{\rightarrow}}

\newcommand{\sG}{\mathscr{G}}

\let\catsymbfont\mathcal

\newcommand{\aC}{{\catsymbfont{C}}}
\newcommand{\aD}{{\catsymbfont{D}}}

\newcommand{\aI}{{\catsymbfont{I}}}
\newcommand{\aJ}{{\catsymbfont{J}}}

\newcommand{\aO}{{\catsymbfont{O}}}

\newcommand{\aV}{{\catsymbfont{V}}}

\newcommand{\bL}{{\mathbb{L}}}

\newcommand{\Del}{\mathbf{\Delta}}
\newcommand{\Adj}{\mathbf{Adj}}

\newcommand{\Cat}{\mathbf{Cat}}

\newcommand{\hCat}{\mathbf{Cat}_h}

\def\quickop#1{\expandafter\DeclareMathOperator\csname
#1\endcsname{#1}}
\quickop{id}\quickop{colim}\quickop{hocolim}\quickop{op}
\quickop{co}\quickop{Ar}\quickop{Hom}\quickop{Ho}
\quickop{ob}\quickop{diag}\quickop{perf}\quickop{Sp}\quickop{Mod}\quickop{Wald}
\quickop{Ext}\quickop{Ob}
\quickop{Tot}\quickop{Lan}\quickop{Map}
\quickop{sk}\quickop{cod}\quickop{dom}
\quickop{holim}\quickop{Stab}\quickop{Comm}\quickop{LMod}\quickop{hofib}\quickop{cf}

\newcommand{\uSq}{\underline{\mathrm{Sq}}}
\newcommand{\uC}{\underline{\mathcal{C}}}

\numberwithin{equation}{section}

\newtheorem{thm}[equation]{Theorem}
\newtheorem*{thm*}{Theorem}
\newtheorem{cor}[equation]{Corollary}
\newtheorem{lem}[equation]{Lemma}
\newtheorem{prop}[equation]{Proposition}
\theoremstyle{definition}
\newtheorem{defn}[equation]{Definition}

\theoremstyle{remark}
\newtheorem{rem}[equation]{Remark}
\newtheorem{example}[equation]{Example}
\newtheorem{aside}[equation]{Aside}

\xyoption{arrow}
\xyoption{curve}
\xyoption{matrix}
\xyoption{cmtip}
\xyoption{frame}
\SelectTips{cm}{}

\newdir{ >}{{}*!/-5pt/\dir{>}}

\begin{document}

\title{Homotopical resolutions associated to deformable adjunctions}

\author{Andrew J. Blumberg}
\address{Department of Mathematics, The University of Texas,
Austin, TX \ 78712}
\email{blumberg@math.utexas.edu}
\thanks{The first author was supported in part by NSF grant DMS-0111298}
\author{Emily Riehl}
\address{Department of Mathematics, Harvard University,
Cambridge, MA \ 02138}
\email{eriehl@math.harvard.edu}
\thanks{The second author was supported in part by an NSF postdoctoral fellowship}


\begin{abstract}    
Given an adjunction $F \dashv G$ connecting reasonable categories with
weak equivalences, we define a new derived bar and cobar construction
associated to the adjunction.  This yields homotopical models of the
completion and cocompletion associated to the monad and comonad of the
adjunction.  We discuss applications of these resolutions to spectral
sequences for derived completions and Goodwillie calculus in general
model categories.
\end{abstract}

\maketitle
\tableofcontents

\section{Introduction}

Bar and cobar resolutions are ubiquitous in modern homotopy theory.
For instance, completions of spaces and spectra with respect to
homology theories and free resolutions of operadic algebras are
examples of this kind of construction.  Formally, these bar or cobar
constructions are associated to the monad or comonad of an adjunction
$F \colon \aC \rightleftarrows \aD \colon G$.  Typically, the
categories $\aC$ and $\aD$ are \emph{homotopical}: equipped with
some well-behaved notion of weak equivalence.  Even in the best cases,
e.g., supposing that $F \dashv G$ is a simplicial Quillen adjunction
between combinatorial simplicial model categories, an immediate
problem that arises is that the comonad $FG$ is not homotopically
well-behaved.  Restricting $G$ to the fibrant-cofibrant objects of
$\aD$ does not help, as there is no reason for their images to be
cofibrant; hence, the composite $FG$ will not preserve weak
equivalences in general.  One could replace $F$ and $G$ by their
point-set derived functors $FQ$ and $GR$, where $Q$ and $R$ denote
cofibrant and fibrant replacement functors on $\aC$ and $\aD$
respectively.  But now the endofunctor $FQGR$ is not obviously a
comonad and so the cobar construction gives a cosimplicial object only
up to homotopy.  (The dual problem arises for the bar construction.)

In applications in the literature, this problem is typically
circumvented by working in situations where $\aC$ has all objects
cofibrant and $\aD$ has all objects fibrant.  However, it is not
always possible to arrange for these conditions to hold.  In this
paper, we describe an approach that completely resolves the coherence
problem on the point set level by working with cofibrant and fibrant
replacement functors which are themselves comonads and monads
respectively. An important precursor to this approach is that these
hypotheses are reasonable: on account of an ``algebraic'' small object
argument due to Richard Garner, cofibrantly generated model categories
admit such functors.  Consequently, for any Quillen adjunction
$F \dashv G$ between cofibrantly generated model categories, an
example of what we call a \emph{deformable adjunction}, there are
point-set level natural transformations that make the functors $FQGR$
and $GRFQ$ into a comonad and a monad ``up to homotopy,'' in a sense
we define below. More precisely, this data enables us to define a pair
of dual natural transformations $\iota \colon Q \to QGRFQ$ and
$\pi \colon RFQGR \to R$ which facilitate the construction of the
derived (co)bar construction as a (co)simplicial object on the point
set level. This extends work of
Radulescu-Banu~\cite{radulescu-banu-thesis} on derived completion, as
we discuss in section~\ref{sec:completions}. Related observations have
been made independently by Arone and Ching in the special case where
all objects of $\aD$ are fibrant, in which case the monad $R$ can be
taken to be the identity \cite{AroneChing2}. In this special case, the
``up to homotopy'' comonad $FQG$ becomes a comonad on the nose.

The monad and comonad resolutions associated to an adjunction  are
formally dual. A classical categorical observation tells a richer
story. There is a free (strict) 2-category $\Adj$ containing an
adjunction. The image of the 2-functor $\Adj \to \Cat$ extending the
adjunction $F \colon \aC \rightleftarrows \aD \colon G$ consists of
the monad and comonad resolutions in the functor categories
$\aC^{\aC}$ and $\aD^{\aD}$ together with a dual pair of split
augmented simplicial objects  in $\aC^{\aD}$ and $\aD^{\aC}$. The fact
that these resolutions assemble into a 2-functor says that, e.g., that
the image of the comonad resolution under $G$ is an augmented
simplicial object in $\aC^{\aD}$ that admits ``extra degeneracies.''
Our constructions produce similar data in the 2-category $\hCat$ of
categories with weak equivalences and weak equivalence preserving
functors. However, the 2-functor $\Adj \to \hCat$ is not strict;
rather it is a coherent mixture of lax and oplax in a sense that will
be described in section~\ref{sec:adjdata}. 

Our primary interest in these results stems from their applications. A
unifying theme is that previous work that required both adjoint
functors to preserve weak equivalences can now be extended to
arbitrary Quillen adjunctions between cofibrantly generated model
categories. In section~\ref{sec:completions}, we generalize the work
of Bousfield on spectral sequences for monadic completions using a new
model of derived completion given by our derived cobar construction.
A closely related issue is the connection to Quillen homology and
Goodwillie calculus in general model categories, an issue we turn to
in section~\ref{sec:calculus}.  As explained in Kuhn \cite{kuhn},
Goodwillie calculus can be carried out in fairly general settings, as
long as the formal stabilization of a category $\aC$ can be described.
However, a key part of understanding the derivatives of the identity
has to do with the completion with respect to the monad of the
accompanying $\Sigma^\infty\dashv \Omega^\infty$ adjunction.  Our
results allow us to describe this completion in any cofibrantly
generated model category.  As a representative application, we
generalize one of the main technical tools used in the Arone-Ching
approach to the chain rule~\cite{AroneChing}.

Sections \ref{sec:htpical} and \ref{sec:background} present background
material on homotopical functors and the algebraic small object
argument. In section \ref{sec:htpyresol}, we introduce the homotopical
resolutions around which our applications will center. In
section \ref{sec:simplicial} we prove a result that may be of
independent interest: that a cofibrantly generated simplicial model
category admits a simplicially enriched fibrant replacement monad and
cofibrant replacement comonad. This extends observations made
by \cite{RSS} and others.  Because later sections presuppose
simplicial enrichments, the appendix~\ref{appendix:simp} reviews how
to use the technology of \cite{RSS} and \cite{dugger} to replace a
Quillen adjunction between reasonable model categories by a simplicial
Quillen adjunction between simplicial model categories.

\subsection*{Acknowledgements}

The authors would like to thank Sam Isaacson, Mike Mandell, Haynes
Miller, and Dominic Verity for helpful conversations. The first author was supported in part by NSF grants
DMS-0906105 and DMS-1151577. The second author was supported in part by an NSF postdoctoral fellowship~DMS-1103790.

\section{Homotopical categories and derived functors}\label{sec:htpical}

For the reader's convenience, we briefly review a few insights presented in \cite{DHKS} that axiomatize and generalize the construction of derived functors between model categories. The authors make a persuasive case that their \emph{homotopical categories} are a good setting to understand derived functors. The familiar construction of derived Quillen functors generalizes seamlessly and this extra generality is useful: derived colimit and limit functors, more commonly referred to as homotopy colimits and limits, are frequently desired for diagram categories that might not admit appropriate model structures.

A \emph{homotopical category} is a category equipped with some class of weak equivalences satisfying the 2-of-3 property, or occasionally, as in \cite{DHKS}, the stronger 2-of-6 property. A homotopical category $\aC$ has an associated \emph{homotopy category} $\Ho\aC$ which is the formal localization at this class of morphisms. The associated localization functor $\aC \to \Ho\aC$ is universal among functors with domain $\aC$ that map the weak equivalences to isomorphisms. A functor $F \colon \aC \to \aD$ between homotopical categories is \emph{homotopical} if it preserves weak equivalences; elsewhere in the literature such $F$ are called \emph{homotopy functors}. In this case, $F$ descends to a unique functor $F \colon \Ho\aC \to \Ho\aD$ that commutes with the localization maps.

\subsection*{Derived functors between homotopical categories}

Of course a generic functor $F \colon \aC \to \aD$ between homotopical categories might fail to be homotopical, in which case it is convenient to have a ``closest'' homotopical approximation. Formally, a \emph{total left derived functor} of $F$ is a \emph{right} Kan extension of $\aC \xrightarrow{F} \aD \to \Ho\aD$ along the localization $\aC \to \Ho\aC$. In practice, it is easier to work with functors between the homotopical categories rather than functors between the associated homotopy categories. With this in mind, a (\emph{point-set}) \emph{left derived functor} is a functor $\bL F \colon \aC \to \aD$ together with a natural transformation $\bL F \to F$ so that the composite with the localization $\aD \to \Ho\aD$ is a total left derived functor of $F$. \emph{Total} and \emph{point-set right derived functors} are defined dually;  the total right derived functor is a left Kan extension.

Derived functors need not exist in general but when they do it is often for the following reason. Frequently,  a non-homotopical functor nonetheless preserves all weak equivalences between certain ``good'' objects. Supposing this is the case, if the category admits a reflection into this subcategory of ``good'' objects, then precomposition with the reflection gives a point-set level derived functor. The following definitions make this outline precise.

\begin{defn} A \emph{left deformation} on $\aC$ consists of a functor $Q \colon \aC \to \aC$ together with a natural weak equivalence $q \colon Q \to 1$. A \emph{right deformation} consists of a functor $R \colon \aC \to \aC$ and a natural weak equivalence $r \colon 1 \to R$.
\end{defn}

The notation is meant to suggestion cofibrant and fibrant replacement respectively. It follows from the 2-of-3 property that the functors $Q$ and $R$ are homotopical and therefore induce adjoint equivalences between $\Ho\aC$ and the homotopy category of any full subcategory containing the image of $Q$ and of $R$. 

In the presence of a left deformation $(Q,q)$ on $\aC$, a functor $F \colon \aC \to \aD$ is \emph{left deformable} if $F$ preserves all weak equivalences between objects in the image of $Q$. The key point is that no additional structure is needed to produce point-set derived functors.

\begin{prop} If $F$ is left deformable with respect to $(Q,q)$, then $Fq \colon FQ \to F$ is a point-set left derived functor of $F$.
\end{prop}
\begin{proof} Follow your nose or see \cite[\S 2]{riehlcathtpythy}.
\end{proof}

\begin{aside} Total derived functors constructed via deformations are better behaved categorically than generic derived functors: the total left derived functor of a functor admitting a left deformation is an \emph{absolute right Kan extension}, i.e., is preserved by any functor. It follows, for instance, that if $F$ and $G$ are an adjoint pair of functors admitting a left and right deformation, respectively, then their total left and right derived functors are adjoints \cite{maltsiniotis}. To the authors' surprise, it does not appear to be possible to prove an analogous result without this stronger universal property, or without some specific knowledge of how the derived functors are constructed, which amounts to the same thing.
\end{aside}

\section{The algebraic small object argument}\label{sec:background}

Our constructions of homotopical resolutions exploit a fact that is not well-known: a large class of model structures admit a fibrant replacement monad and a cofibrant replacement comonad. More precisely, any cofibrantly generated model structure on a category that \emph{permits the small object argument}, a choice of two set-theoretical conditions described below, has a fibrant replacement monad and a cofibrant replacement comonad. These are constructed by a variant of Quillen's small object argument due to Richard Garner. Our main results exploit the following immediate corollary to his Theorem \ref{thm:garner} below.

\begin{cor}\label{cor:main} Suppose $\aC$ is a model category that permits the small object argument equipped with a set $\aI$ of generating cofibrations and a set $\aJ$ of trivial cofibrations that detect fibrant objects, in the sense that an object is fibrant if and only if it lifts against the elects of $\aJ$. Then there is a monad $(R, r \colon 1 \to R, \mu \colon R^2 \to R)$ and a comonad $(Q, q \colon Q \to 1, \delta \colon Q \to Q^2)$ on $\aC$ such that for all objects $X$ \begin{enumerate} \item $RX$ is fibrant and $QX$ is cofibrant \item $r_X \colon X \to RX$ is a trivial cofibration and $q_X \colon QX \to X$ is a trivial fibration \item $\mu_X \colon R^2X \to RX$ and $\delta_X \colon QX \to Q^2X$ are weak equivalences  \end{enumerate} Furthermore, $R$ and $Q$ are homotopical and hence descend to a monad and comonad on $\Ho\aC$.
\end{cor}
\begin{proof}
(i) and (ii) follow from the construction of Theorem \ref{thm:garner} below. (ii) implies (iii) and the fact that the functors are homotopical by the 2-of-3 property.
\end{proof}

In the more general setting of homotopical categories, we might call $(R,r,\mu)$ a \emph{homotopical monad} and $(Q,q,\delta)$ a \emph{homotopical comonad}, meaning that the natural transformations are weak equivalences and the functors preserve weak equivalences.

\begin{rem}
Of course, a set of generating trivial cofibrations would suffice for $\aJ$, but for certain examples the weaker hypothesis is preferred.  For instance, the inner horn inclusions detect fibrant objects for Joyal's model structure on simplicial sets: an $\infty$-category is precisely a simplicial set that has the extension property with respect to this set of arrows. Algebras for the fibrant replacement monad they generate are ``algebraic $\infty$-categories,'' i.e., $\infty$-categories with a specified filler for every inner horn. By contrast, an explicit set of generating trivial cofibrations is not known, though it can be proven to exist \cite[\S A.2.6]{HTT}.
\end{rem}

\subsection*{Garner's small object argument} 

The monad $R$ and comonad $Q$ are constructed using a refinement of Quillen's small object argument due to Richard Garner \cite{GarnerUSOA}. Given a set of arrows $\aI$ in a category satisfying a certain set-theoretical condition, Quillen's small object argument constructs a functorial factorization such that the right factor of any map satisfies the right lifting property with respect to $\aI$ and such that the left factor is a relative $\aI$-cell complex, i.e., is a transfinite composite of pushouts of coproducts of maps in $\aI$. 

Garner's small object argument constructs the functorial factorization in such a way that the functor sending a map to its right factor is a monad on the category of arrows and the functor sending a map to its left factor is a comonad. The right factor of some map is a (free) algebra for this monad, which implies that it lifts against the generating arrows. The left factor of some map is a (free) coalgebra for the comonad; this is analogous to the usual cellularity condition. Because the comonad and monad extend a common functorial factorization, it is easy to prove that  the left factor lifts against any map the lifts against the generating arrows. 

These results require that the construction described below \emph{converges}. The set theoretical hypotheses encoded by our phrase ``permits the small object argument'' are designed to guarantee that this is the case.  

\begin{thm}[Garner]\label{thm:garner} Let $\aC$ be a cocomplete category satisfying either of the following conditions: \begin{itemize} \item[$(*)$] Every $X \in \aC$ is $\kappa_X$-presentable for some regular cardinal $\kappa_X$.   \item[$(\dagger)$] Every $X \in \aC$ is $\kappa_X$-bounded with respect to some well-copowered orthogonal factorization system on $\aC$ for some regular cardinal $\kappa_X$. \end{itemize} Then any set of arrows generates a functorial factorization \[\xymatrix{ X \ar[rr]^f \ar[dr]_{Cf} & & Y \\ & Ef \ar[ur]_{Ff}}\]
where the left-hand functor $C$ is a comonad and the right-hand functor $F$ is a monad on the category of arrows in $\aC$.
\end{thm}

We say a category $\aC$ \emph{permits the small object argument} if it is cocomplete and satisfies either $(*)$ or $(\dagger)$. Locally presentable categories such as {\bf sSet} satisfy $(*)$. Locally bounded categories such as {\bf Top}, {\bf Haus}, and {\bf TopGp}   satisfy $(\dagger)$. We don't know of a category that permits Quillen's small object argument but fails to satisfy these conditions; hence, when we refer to a cofibrantly generated model category, we tacitly suppose that the category permits the small object argument in this sense.

The functors constructed by Garner's small object argument satisfy two universal properties, which are instrumental in describing the applications of this small object argument to model categories \cite{Riehl}. One ensures that algebras for the monad are precisely those arrows which have the right lifting property with respect to each of the generators. The other states that the coalgebra structures assigned the generating arrows are initial in some precise sense. 

\begin{aside} In fact, Garner's construction works for any small \emph{category} of arrows. Morphisms between the generating arrows can be used to encode coherence requirements for the lifting properties characterizing the right class. In particular, there exist model categories that are provably not cofibrantly generated (in the classical sense) whose cofibrations and trivial cofibrations are generated by categories of arrows \cite[\S 4]{Riehl}.
\end{aside}

Garner's construction begins in the same way as Quillen's: to factor an arrow $f$, first form the coproduct of arrows $i \in \aI$ indexed by commutative squares from $i$ to $f$.  Then push out this coproduct along the canonical map from its domain to the domain of $f$. This defines  functors $C_1$ and $F_1$ that give the \emph{step-one factorization} of $f$ displayed below.
\begin{equation}\label{eq:stepone}\xymatrix{ \cdot \ar[d]_{\coprod\limits_{i \in \aI} \coprod\limits_{\mathrm{Sq}(i,f)}i} \ar[r] \ar@{}[dr]|(.8){\ulcorner}& \dom f \ar@{=}[r] \ar[d]^{C_1f} & \dom f \ar[d]^f \\ \cdot \ar[r] & E_1f \ar[r]_{F_1f} & \cod f}\end{equation}

Next, repeat this process with $F_1f$ in place of $f$, factoring $F_1f$ as $C_1F_1f$ followed by $F_1F_1f$. The composite $C_1F_1f\cdot C_1 f$ is the left factor and $F_1F_1f$ is the right factor in Quillen's step-two factorization. By contrast, the left factor $C_2f$ in Garner's step-two factorization is a quotient of Quillen's left factor, defined by  the coequalizer:
\begin{equation}\label{eq:steptwo}\xymatrix@C=40pt{ \dom f \ar[d]_{C_1f} \ar@{=}[r] & \dom f \ar[d]^{C_1F_1f \cdot C_1f} \ar@{=}[r] & \dom f \ar@{-->}[d]^{C_2f} \ar@{=}[r] & \dom f \ar[d]^f \\ E_1 f \ar@/_3pc/[rrr]_{F_1f} \ar@<.5ex>[r]^{C_1F_1f} \ar@<-.5ex>[r]_(.55){E_1(C_1f,1)} & E_1 F_1f  \ar@/_1pc/[rr]_{F_1F_1f} \ar@{-->}[r] & E_2f \ar@{-->}[r]^{F_2f} & \cod f}\end{equation} Here $C_2f$ is defined to be the coequalizer in the arrow category of the pair of maps from $C_1f$ to Quillen's step-two left factor. The right factor $F_2f$ is defined via the universal property of the coequalizer.

As is the case for Quillen's small object argument, the step-two left factor is in the weakly saturated class $\aI$-cof generated by $\aI$, though this is not obvious. The reason is that $C_1f$ and $C_1F_1f \cdot C_1f$ are canonically coalgebras for the comonad $C$ that will be produced at the termination of Garner's small object argument. The maps in the coequalizer diagram are maps of $C$-coalgebras. Hence, the colimit $C_2f$ is canonically a $C$-coalgebra and hence a member of $\aI$-cof. 

For any $\alpha$ that is the successor of a successor ordinal, the step-$\alpha$ factorization is defined analogously: $C_\alpha$ is a quotient of  $C_1F_{\alpha -1}\cdot C_{\alpha-1}$.  When $\alpha$ is a limit ordinal, the step-$\alpha$ left factor is a quotient of the colimit of the previous left factors. The factorizations for successors of limit ordinals are similarly constructed by a coequalizer. See \cite{GarnerUSOA} for more details.

Garner proves that under the hypotheses of Theorem \ref{thm:garner} this process \emph{converges}; in other words, there is no need to specify an artificial halting point. In his proof, he shows that this construction coincides with a known procedure for forming a free monad on the pointed endofunctor $F_1$; hence, the right factor is a monad and the resulting factorization satisfies the universal properties mentioned above.  

\begin{example} When the cofibrations are monomorphisms, there is a simpler description of this construction. A prototypical example is given by the generators $\aI = \{ \partial\Delta^n \to \Delta^n\}$ in the category of simplicial sets. The functor $C_1$ attaches simplices to fill every sphere in the domain of $f$. The functor $C_1F_1$ again attaches fillers for all spheres in this new space---including those which were filled in step one. The coequalizer (\ref{eq:steptwo}) identifies the simplices attached to the same spheres in step one and step two. So the effect of Garner's construction is that in step two the only spheres which are filled are those which were not filled in step one. The modification to the rest of the construction is similar. In the case where every cofibration is a monomorphism, Garner's small object argument coincides with the construction described in \cite{radulescu-banu-thesis}.
\end{example}

\subsection*{Fibrant and cofibrant replacement} The functorial factorizations of Theorem \ref{thm:garner} can be specialized to construct a fibrant replacement monad $R$ and a cofibrant replacement comonad $Q$. Because the monad $F$ produced by the small object argument preserves codomains, restriction to the subcategory of maps over the terminal object defines a monad on the category $\aC$.  The components of the unit of the monad are the left factors of the maps $X \to *$. When the generators are a set of trivial cofibrations that detect fibrant objects, the resulting monad $R$ has the properties of Corollary \ref{cor:main}. 

The cofibrant replacement comonad is obtained dually. Restricting the domain-preserving comonad $C$ of Theorem \ref{thm:garner} to the subcategory of maps under the initial object produces a comonad on $\aC$. The components of the counit are given by the right factors of the maps $\emptyset \to X$. When this construction is applied to the generating cofibrations, these right factors are trivial fibrations and the comonad $Q$ gives a cofibrant replacement.

\begin{aside} The monad $F$ and comonad $C$ constructed in Theorem \ref{thm:garner} satisfy further compatibility conditions. Provided the generating trivial cofibrations are \emph{cellular} cofibrations---it suffices that these maps are elements of $\aI$-cell---there is a canonical \emph{natural transformation} $RQ \to QR$. By contrast, for generic fibrant and cofibrant replacements, there is always a comparison map $RQX \to QRX$ but it need not be natural. Furthermore this map is distributive law of the monad $R$ over the comonad $Q$. As a formal consequence,  the monad $R$ lifts to a monad on the category of coalgebras for $Q$, defining ``fibrant replacement'' on the category of ``algebraically cofibrant objects.'' Dually, the comonad $Q$ lifts to a comonad on the category of ``algebraically fibrant objects.''  Coalgebras for the lifted comonad  correspond to algebras for the lifted monad, defining a category of ``algebraically fibrant-cofibrant objects''  \cite[\S 3.1]{Riehl}.
\end{aside}

\section{Homotopical resolutions}\label{sec:htpyresol}

Given an adjoint pair of functors $F \colon \aC \rightleftarrows \aD \colon G$, the unit and counit maps,\footnote{A natural transformation $\alpha \colon H \to K$ gives rise to natural transformations $G\alpha \colon GH \to GK$ and $\alpha F \colon HF \to KF$ by a procedure known as ``whiskering'' (assuming these composite functors are well-defined). The components of $G\alpha$ are the images of the components of $\alpha$ under the functor $G$; the components of $\alpha F$ are the components of $\alpha$, restricted to objects in the image of $F$. To save space, and because the domains and codomains of the natural transformation leave no room for ambiguity, we refer to $\alpha F$ as simply $\alpha$.}
 by the triangle identities, give rise to a coaugmented cosimplicial object in $\aC^{\aC}$ 
\[\xymatrix@C=60pt{ 1 \ar[r]|-{\eta} & GF \ar@<2ex>[r]|-{\eta} \ar@<-2ex>[r]|-{GF\eta} & GFGF \ar[l]|-{G\epsilon} \ar@<4ex>[r]|-{\eta} \ar[r]|-{GRF\eta} \ar@<-4ex>[r]|-{GFGF\eta} & \ar@<2ex>[l]|-{GFG\epsilon} \ar@<-2ex>[l]|-{G\epsilon} GFGFGF \cdots}\]
 an augmented simplicial object in $\aD^{\aD}$ 
\[\xymatrix@C=60pt{ 1 & FG \ar[l]|-{\epsilon} \ar[r]|-{F\eta} & FGFG \ar@<-2ex>[l]|-{\epsilon} \ar@<2ex>[l]|-{FG\epsilon} \ar@<2ex>[r]|-{F\eta} \ar@<-2ex>[r]|-{FGF\eta} & \ar@<4ex>[l]|-{FGFG\epsilon} \ar[l]|-{FG\epsilon} \ar@<-4ex>[l]|-{\epsilon} FGFGFG \cdots}\]
and augmented simplicial objects in $\aD^{\aC}$ and $\aC^{\aD}$ admitting forwards and backwards contracting homotopies
\[\xymatrix@C=60pt{ F \ar@<-1ex>[r]|-{F\eta} & FGF \ar@<-1ex>[l]|-{\epsilon} \ar@<1ex>[r]|-{F\eta} \ar@<-3ex>[r]|-{FGF\eta} & FGFGF \ar@<-3ex>[l]|-{\epsilon} \ar@<1ex>[l]|-{FG\epsilon} \ar@<3ex>[r]|-{F\eta} \ar@<-1ex>[r]|-{FGF\eta} \ar@<-5ex>[r]|-{FGFGF\eta} & \ar@<3ex>[l]|-{FGFG\epsilon} \ar@<-1ex>[l]|-{FG\epsilon} \ar@<-5ex>[l]|-{\epsilon} FGFGFGF \cdots}\]
\[\xymatrix@C=60pt{ G \ar@<1ex>[r]|-{\eta} & GFG \ar@<3ex>[r]|-{\eta} \ar@<-1ex>[r]|-{GF\eta} \ar@<1ex>[l]|-{G\epsilon} & GFGFG \ar@<-1ex>[l]|-{G\epsilon} \ar@<5ex>[r]|-{\eta} \ar@<1ex>[r]|-{GRF\eta} \ar@<-3ex>[r]|-{GFGF\eta} \ar@<3ex>[l]|-{GFG\epsilon} & \ar@<1ex>[l]|-{GFG\epsilon} \ar@<-3ex>[l]|-{G\epsilon} \ar@<5ex>[l]|-{GFGFG\epsilon} GFGFGFG \cdots}\]
Here the contracting homotopies, also called ``splittings'' or ``extra degeneracies,'' are given by the bottom and top $\eta$'s respectively.  It follows that the geometric realization or homotopy invariant realization of the simplicial objects spanned by the objects in the image of $FGF$ and $GFG$ are simplicial homotopy equivalent to $F$ and $G$. Dual results apply to the (homotopy invariant) totalization of the cosimplicial object spanned by these same objects; in this case the ``extra codegeneracies'' are given by the top and bottom $\epsilon$'s.

Collectively, these diagrams display the image of a 2-functor whose domain is $\Adj$, the free 2-category containing an adjunction \cite{schanuelstreet}. More precisely, each diagram is the image of one of the  four hom-categories of this two object 2-category.

Historically, these resolutions have wide applications. But for a generic adjoint pair of functors between homotopical categories, or even for a Quillen adjunction, the functors involved in this construction are not homotopical. In particular, weakly equivalent objects need not induce weakly equivalent (co)simplicial resolutions. The usual fix---replacing $F$ and $G$ with point-set left and right derived functors---inhibits the definition of the structure maps of the (co)simplicial objects. One might attempt to define these maps as zig-zags but the backwards maps won't be weak equivalences. However if the left deformation for $F$ extends to a homotopical comonad and the right deformation for $G$ is a homotopical monad then we shall see that it is possible to constructed derived (co)simplicial resolutions.

\subsection*{Constructing the resolutions}

The following definition encodes our standing hypotheses.

\begin{defn}\label{defn:defadj}
Suppose that $\aC$ and $\aD$ are homotopical categories such that
$\aC$ is equipped with a homotopical comonad $(Q,q,\delta)$ and $\aD$
is equipped with a homotopical monad $(R,r,\mu)$.  In this context, an
adjunction $(F, G, \eta \colon 1 \to GF, \epsilon \colon FG \to 1)$ is
a \emph{deformable adjunction} if $F$ is left deformable with respect
to $(Q,q)$ and $G$ is right deformable with respect to $(R,r)$.
\end{defn}

For instance, these conditions are satisfied if $\aC$ and $\aD$ are cofibrantly generated model categories  and $F \dashv G$ is any Quillen adjunction.  Note, our definition is stronger than the usage of this term in \cite{DHKS} or \cite{Shulman} because we have asked that the deformations extend to (co)monads.

For any deformable adjunction there exists a pair of dual natural transformations $\iota \colon Q \to QGRFQ$ and $\pi \colon RFQGR \to R$ defined by 
\[ \iota \colon Q \stackrel{\delta}{\to} Q^2 \stackrel{Q\eta}{\to} QGFQ \stackrel{QGr}{\to} QGRFQ\] and \[ \pi \colon RFQGR \stackrel{RFq}{\to} RFGR \stackrel{R\epsilon}{\to} R^2 \stackrel{\mu}{\to} R\] which satisfy the following conditions:

\begin{lem}\label{lem:simpid1} The composites \[RFQ \stackrel{RF\iota}{\to} RFQGRFQ \stackrel{\pi}{\to} RFQ\] and \[QGR \stackrel{\iota}{\to} QGRFQGR \stackrel{QG\pi}{\to} QGR\] are identities.
\end{lem}
\begin{proof} These statements are dual so a single diagram chase suffices
\begin{equation}\label{eq:simpid1}\xymatrix@R=40pt@C=80pt@!0{RFQ \ar[r]^-{RF\delta} \ar[dr]_1 & RFQ^2 \ar[r]^-{RFQ\eta} \ar[d]^{RFq} & RFQGFQ \ar[r]^-{RFQGr} \ar[d]^{RFq} & RFQGRFQ \ar[d]^{RFq} \\ &  RFQ \ar[dr]_1 \ar[r]^-{RF\eta} & RFGFQ \ar[r]^-{RFGr} \ar[d]^{R\epsilon} & RFGRFQ \ar[d]^-{R\epsilon} \\ &&RFQ \ar[r]^-{Rr} \ar[dr]_1 & R^2FQ \ar[d]^{\mu} \\ &&& RFQ}\end{equation} The squares commute by naturality of the vertical arrow; the triangles commute by a comonad, adjunction, and monad triangle identity, respectively.
\end{proof}

\begin{lem}\label{lem:simpid2} The map $\pi \colon RFQGR \to R$ is
associative in the sense that 
\[
\pi \cdot \pi_{FQGR}=\pi \cdot RFQG\pi.
\]
Dually, $\iota \colon Q \to QGRFQ$ is coassociative.  
\end{lem}
\begin{proof} Again, a single diagram chase suffices:
\begin{equation}\label{eq:simpid2}\xymatrix{RFQGRFQGR \ar[r]^-{RFq} \ar[d]_{RFQGRFq} & RFGRFQGR \ar[r]^-{R\epsilon} \ar[d]_{RFGRFq} & R^2FQGR \ar[r]^-\mu \ar[d]_{R^2Fq} & RFQGR \ar[d]^{RFq} \\  RFQGRFGR \ar[d]_{RFQGR\epsilon} \ar[r]_-{RFq} & RFGRFGR \ar[r]_-{R\epsilon} \ar[d]_{RFGR\epsilon} & R^2FGR \ar[r]_-{\mu} \ar[d]_{R^2\epsilon} & RFGR \ar[d]^{R\epsilon} \\ RFQGR^2 \ar[d]_{RFQG\mu} \ar[r]_-{RFq} & RFGR^2 \ar[d]_{RFG\mu} \ar[r]_-{R\epsilon} & R^3 \ar[d]_{R\mu} \ar[r]_-{\mu} & R^2 \ar[d]^\mu \\ RFQGR \ar[r]_-{RFq} & RFGR \ar[r]_-{R\epsilon} & R^2 \ar[r]_-{\mu} & R}\end{equation} The bottom-right square expresses the associativity of the monad product $\mu$; the remaining squares commute by naturality of the horizontal map. 
\end{proof}

The relations presented in Lemmas \ref{lem:simpid1} and \ref{lem:simpid2} say precisely that $\iota$ and $\pi$ can be used to form cosimplicial resolutions. 

\begin{cor}\label{cor:resolutions} Given $F,G,Q,R$ as above, there is an coaugmented cosimplicial object in $\aC^{\aC}$
\begin{equation}\label{eq:cosimpresol}\xymatrix@C=60pt{ Q \ar[r]|-{\iota} & QGRFQ \ar@<2ex>[r]|-{\iota} \ar@<-2ex>[r]|-{QGRF\iota} & QGRFQGRFQ \ar[l]|-{QG\pi} \ar@<4ex>[r]|-{\iota} \ar[r]|-{QGRF\iota} \ar@<-4ex>[r]|-{QGRFQGRF\iota} & \ar@<2ex>[l]|-{QGRFQG\pi} \ar@<-2ex>[l]|-{QG\pi} \cdots}\end{equation}
and dually an augmented simplicial object in $\aD^{\aD}$
\[\xymatrix@C=60pt{ R & RFQGR \ar[l]|-{\pi} \ar[r]|-{RF\iota} & RFQGRFQGR \ar@<-2ex>[l]|-{\pi} \ar@<2ex>[l]|-{RFQG\pi} \ar@<2ex>[r]|-{RF\iota} \ar@<-2ex>[r]|-{RFQGRF\iota} & \ar@<4ex>[l]|-{RFQGRFQG\pi} \ar[l]|-{RFQG\pi} \ar@<-4ex>[l]|-{\pi} \cdots}\] Furthermore, we have augmented simplicial objects in $\aD^{\aC}$ and $\aC^{\aD}$ with forwards and backwards contracting homotopies
\[\xymatrix@C=60pt{ RFQ  \ar@<-1ex>[r]|-{RF\iota} & RFQGRFQ \ar@<-1ex>[l]|-{\pi} \ar@<1ex>[r]|-{RF\iota} \ar@<-3ex>[r]|-{RFQGRF\iota} & RFQGRFQGRFQ \ar@<-3ex>[l]|-{\pi} \ar@<1ex>[l]|-{RFQG\pi} \cdots}\] 
\[\xymatrix@C=60pt{ QGR \ar@<1ex>[r]|-{\iota} & QGRFQGR \ar@<3ex>[r]|-{\iota} \ar@<-1ex>[r]|-{QGRF\iota}  \ar@<1ex>[l]|-{QG\pi} & QGRFQGRFQGR \ar@<-1ex>[l]|-{QG\pi} \ar@<3ex>[l]|-{QGRFQG\pi}  \cdots }\]
\end{cor}

If $R$ and $Q$ are fibrant and cofibrant replacements satisfying the hypotheses of Corollary \ref{cor:main} and $F \dashv G$ a Quillen adjunction, the functors comprising these (co)simplicial objects restrict to maps between $\aC^{\cf}$ and $\aD^{\cf}$, the full subcategories spanned by the fibrant-cofibrant objects. 

\begin{rem} While the diagrams of Corollary \ref{cor:resolutions} collectively define the images of the four hom-categories in $\Adj$, this data does not quite assemble into a strict 2-functor with domain $\Adj$. We'll have more to say about this point in section~\ref{sec:adjdata}.
\end{rem}

\subsection*{Algebras and coalgebras}

By Lemmas \ref{lem:simpid1} and \ref{lem:simpid2}, the triples $(GRFQ,  \iota, G\pi)$ and $(FQGR,  \pi, F\iota)$ nearly define a point-set level homotopical monad and comonad. More precisely, the zig-zags \[ 1 \xleftarrow{q} Q \xrightarrow{\iota} QGRFQ \xrightarrow{q} GRFQ \qquad FQGR \xrightarrow{r} RFQGR \xrightarrow{\pi} R \xleftarrow{r} 1 \] represent unit and counit maps that make $GRFQ$ into a monad on $\Ho\aC$ and $FQGR$ into a comonad on $\Ho\aD$. These observations suggest the following definitions.

\begin{defn} Write $T$ for the composite $GRFQ$. A \emph{homotopy} $T$-\emph{algebra} is an object $X \in \aC$ together with a map $h \colon TX \to X$ so that $Qh \cdot \iota = \id$ and $h \cdot G\pi = h \cdot Th$. A map $f \colon (X,h) \to (X',h')$ of homotopy $T$-algebras is a map $f \colon X \to X'$ commuting with the action maps.
\end{defn}

 These definitions are designed to produce a simplicial object in $\aD$.
\begin{equation}\label{eq:simpobj}\xymatrix@C=60pt{ RFQX \ar[r]|-{RF\iota} & RFQGRFQX \ar@<2ex>[l]|-{RFQh}\ar@<-2ex>[l]|-{\pi} \ar@<2ex>[r]|-{RF\iota} \ar@<-2ex>[r]|-{RFQGRF\iota} & RFQGRFQGRFQX \ar[l]|-{RFQG\pi} \ar@<4ex>[l]|-{RFQGRFQh} \ar@<-4ex>[l]|-{\pi} \cdots}\end{equation}
Applying the functor $G$, the simplicial object (\ref{eq:simpobj}) admits an augmentation to $X$. Further applying the functor $Q$, (\ref{eq:simpobj}) admits a backwards contracting homotopy. Algebra maps give rise to natural transformations between these diagrams. By standard arguments, the simplicial object $QT_\bullet X$ is simplicially homotopically equivalent to the constant simplicial object at $QX$, and consequently its geometric realization is weakly equivalent to $QX$, and hence to $X$. 

\begin{example}
Any object $X$ has an associated \emph{free homotopy} $T$-\emph{algebra} $(TX,G\pi)$. By naturality of $\pi$, for any map $f \colon X \to X'$, $Tf$ is a map of homotopy $T$-algebras $(TX,G\pi) \to (TX', G\pi)$. 
\end{example}

Let $\aC^T$ be the category of homotopy $T$-algebras. We regard it as a homotopical category with weak equivalences created by the forgetful functor $U \colon \aC^T \to \aC$. Write $T \colon \aC \to \aC^T$ to denote the free-algebra functor. Note that the composite of this functor with the forgetful functor is the original $T$. Because $T \colon \aC \to \aC$ is homotopical, $T \colon \aC \to \aC^T$ is homotopical. More generally, for any $Y \in \aD$, there is an associated $T$-algebra $(GRY, G\pi)$. This defines a homotopical functor  $GR \colon \aD \to \aC^T$ that lifts $GR \colon \aD \to \aC$ through the category of homotopy $T$-algebras.

Coalgebras for $FQGR$ can be defined dually.

\begin{rem} In the special case where $F$ is homotopical, $Q$ may be taken to be the identity deformation, and the triple $(GRF, \iota, G\pi)$ defines an actual monad on the point set level. In this case our notion of homotopy algebra coincides with the standard notion of algebra for the monad $GRF$.
\end{rem}

\section{Homotopy spectral sequences for monadic
completions}\label{sec:completions} 

For a commutative ring $R$, Bousfield and Kan define the
$R$-completion of a simplicial set $X$ as the $\Tot$ of the
cosimplicial resolution of the monad associated to the free $R$-module
functor (modulo the basepoint)~\cite{BK}.  This construction is
homotopically well-behaved because the monad in question preserves all
weak equivalences and the resulting cosimplicial object is Reedy
fibrant.  When generalizing to completions associated to other
adjunctions, we can handle the latter issue by using the homotopy
limit rather than the underived $\Tot$, but homotopical control on the
monad is a more serious problem.  This is precisely the problem
that our constructions solve.  We suppose in this section that we are
given  a deformable adjunction $F \colon \aC \rightleftarrows \aD \colon G$ in the sense of \ref{defn:defadj}.

\begin{defn}\label{def:dercomp} 
Writing $T=GRFQ$, the derived completion $\hat{X}$ of an object $X$ in
$\aC$ is the $\Tot$ of a Reedy fibrant replacement of the cosimplicial
object
\[\xymatrix@C=60pt{
QTX \ar@<2ex>[r]|-{\iota} \ar@<-2ex>[r]|-{QGRF\iota} & QT^2
X \ar[l]|-{QG\pi} \ar@<4ex>[r]|-{\iota} \ar[r]|-{QGRF\iota} \ar@<-4ex>[r]|-{QGRFQGRF\iota}
& QT^3 X \ar@<-2ex>[l]|-{QG\pi} \ar@<2ex>[l]|-{QGRFQG\pi} \cdots}\]
\end{defn}

Corollary \ref{cor:resolutions} implies that the derived
completion is well-defined and has  a natural zig-zag augmentation
\[
\xymatrix{
X & QX \ar[l] \ar[r] & \hat{X}
}
\]
induced by $q$ and $\iota$.  By construction, the derived
completion is functorial and preserves weak equivalences:

\begin{prop}\label{prop:dercompworks}
The derived completion $X \mapsto \hat{X}$ defines a functor
$\aC \to \aC$ such that if $X \to Y$ is a weak equivalence in $\aC$
then the induced map 
\[
\hat{X} \to \hat{Y}
\]
is a weak equivalence.
\end{prop}

Control on this kind of completion is a central technical issue in the
Hess-Harper work on homotopy completions and Quillen
homology~\cite{HarperHess}.  In particular,
Proposition~\ref{prop:dercompworks} gives a direct approach to the
rigidification technology applied to produce the ``TQ-completion''
in~\cite[3.15]{HarperHess}.  This construction also encompasses
algebraic completions in the setting of ring spectra.  For instance,
given a commutative $S$-algebra $A$ and a commutative $A$-algebra
$B$, the derived completion of an $A$-module $M$ with respect to the
homotopical adjunction with left adjoint $B \sma_A -$ and right
adjoint the restriction recovers the derived completion of
Carlsson~\cite[3.1]{carlssonderived}.  Whereas Carlsson maintains
homotopical control by working in the EKMM category of spectra (where
all objects are fibrant), Proposition~\ref{prop:dercompworks} applies
in all of the modern categories of spectra.

For a pointed simplicial model category $\aC$, we obtain a Bousfield-Kan
homotopy spectral sequence computing the homotopy groups of $\hat{X}$,
which arises as the usual homotopy limit spectral sequence associated
to a cosimplicial object. See~\cite[\S 2.9]{Bousfield} for a
modern summary of this procedure.  

\begin{cor}\label{cor:spec}
For any object $W$ in $\aC$, there is a homotopy limit spectral
sequence associated to a Reedy fibrant replacement of the cosimplicial
object 
\[
\Map_{\aC}(W,QT^{\bullet} X)
\]
defined using the Dwyer-Kan mapping complex in $\aC$.
\end{cor}

More generally, we can apply our theory to extend the range of applicability of Bousfield's
comprehensive treatment of completion spectral sequences in the setting of resolution model structures~\cite{Bousfield}.  Bousfield works in the setting of a left proper simplicial model category $\aC$
equipped with a suitable class $\sG$ of injective models. Provided that all objects in $\aC$ admit $\sG$-resolutions, he constructs homotopy spectral sequences computing the $\sG$-completion.  We recall
his definitions from ~\cite[\S 3.1]{Bousfield} here.

Let $\sG$ be a class of group objects in $\Ho\aC$.  A map $i \colon A \to B$ in $\Ho\aC$ is $\sG$-\emph{monic} when $i^* \colon [B,\Omega^nG] \to [A,\Omega^nG]$ is onto for each $G \in \sG$ and $n \geq 0$.  An object
$Y \in \Ho\aC$ is called $\sG$-\emph{injective} when $i^* \colon [B,\Omega^nY] \to [A,\Omega^nY]$ is onto for each $\sG$-monic map $A \to B$ in $\Ho\aC$ and $n \geq 0$.  We say that $\Ho\aC$ \emph{has enough}
$\sG$-\emph{injectives} when each object in $\Ho\aC$ is the source of a $\sG$-monic map to an $\sG$-injective target, in which case $\sG$ is a class of \emph{injective models} in $\Ho\aC$.  In this setting Bousfield defines $\sG$-completions and $\sG$-homotopy spectral sequences.

In particular, Bousfield shows that given a
monad $T$ that preserves weak equivalences and has the further
properties that 
\begin{enumerate}
\item $TX$ is a group object in $ \Ho\aC$ and
\item $\Omega TX$ is $T$-injective in $\Ho\aC$, 
\end{enumerate}
then the completion with respect to the monad $T$---i.e., the class of
injective models specified by $\{TX\mid X \in \Ho\aC\}$---fits into his
framework \cite[\S 7.5]{Bousfield} and gives rise to a homotopy
spectral sequence \cite[\S 5.8]{Bousfield}.  He observes that the
monads associated to Quillen adjunctions $F \colon \aC \to \aD$ and $G
\colon \aD \to \aC$ where all objects in $\aC$ are cofibrant and all
objects in $\aD$ are fibrant satisfy these requirements.

Our work allows us to extend this to any deformable adjunction:

\begin{thm}\label{thm:holimspec}
Let $\aC$ and $\aD$ be cofibrantly generated model
 categories such that $\aC$ is left proper, pointed, and simplicial, and let
\[
F \colon \aC \rightleftarrows \aD \colon G
\] 
be a Quillen adjunction.  Writing $T=GRFQ$, suppose that
\begin{enumerate}
\item $T X$ is a group object in $\Ho\aC$ and
\item $\Omega TX$ is $T$-injective in $\Ho\aC$.
\end{enumerate}
Then $\sG = \{T X \mid X \in \Ho\aC\}$ forms an injective class
and for any object $W$ in $\aC$ there is a homotopy spectral sequence
associated to the filtration on 
\[
\Map_{\aC}(W, QT^\bullet X)
\]
computing the $W$-relative homotopy groups of the derived completion
$\hat{X}$.
\end{thm}

\begin{proof}
Following the outline in~\cite[7.5]{Bousfield}, the argument
of~\cite[7.4]{Bousfield} applies verbatim to show that the
cosimplicial object $QT^\bullet X$ provides a weak resolution of
$QX$ with respect to the injective class $\{T X\}$.  That is, regarding $QX$ as a constant cosimplicial object, 
the map $QX \to QT^\bullet X$  is a $\sG$-equivalence and $QT^\bullet X$ is
levelwise $\sG$-injective.  This implies by~\cite[6.5]{Bousfield} that
there is a homotopy spectral sequence for $\widehat{QX}$; since
$\widehat{QX} \to \widehat{X}$ is a weak equivalence, the homotopy
spectral sequence computes the homotopy groups of $\widehat{X}$.   
\end{proof}

\begin{rem}
In the more restricted setting when we have a monad which preserves
weak equivalences between cofibrant objects, Definition~\ref{def:dercomp} recovers Radulescu-Banu's construction of the derived completion and the associated Bousfield-Kan spectral sequence of Corollary~\ref{cor:spec} recovers his construction of the homotopy spectral sequence of the completion~\cite{radulescu-banu-thesis}.  
\end{rem}

\section{Simplicially enriched fibrant-cofibrant replacement}\label{sec:simplicial}

A particular application of interest of the derived completion is to the adjunction $\Sigma^{\infty}\dashv \Omega^\infty$ connecting a suitable model category $\aC$ and its stabilization $\Stab(\aC)$. This will be the subject of section~\ref{sec:calculus} where we will discuss a generalization of certain technical lemmas of Goodwillie calculus enabled by our results of section~\ref{sec:htpyresol}.

Functors suitable for the analysis provided by Goodwillie calculus are simplicially enriched, homotopical, and frequently also \emph{reduced}, meaning they preserve the basepoint up to weak equivalence. An important preliminary to the results of section~\ref{sec:calculus} is that in a cofibrantly generated simplicial model category, the fibrant replacement monad and cofibrant replacement comonad can be made to satisfy these hypotheses. The second and third of these properties hold for any (co)fibrant replacement on account of the natural weak equivalence to or from the identity. In this section, we will show that the first property can also be made to hold in any simplicial model category. 

The proof given here applies more generally to produce $\aV$-enriched functorial factorizations in any $\aV$-model category for which tensoring with objects in $\aV$ defines a left Quillen functor, as is the case when all objects in $\aV$ are cofibrant \cite[\S 13]{riehlcathtpythy}. 

\begin{thm}\label{thm:enrichedreplacement} A cofibrantly generated simplicial model category admits a simplicially enriched fibrant replacement monad and a simplicially enriched cofibrant replacement comonad.
\end{thm}
\begin{proof}
A modified version of Garner's small object argument, described below, produces a functorial factorization in which the left factor is a simplicially enriched comonad and the right functor is a simplicially enriched monad. The SM7 axiom and the fact that all simplicial sets are cofibrant implies that the factorizations remain appropriate for the model structure. The result follows.

It remains to explain the modification of the small object argument. In a category cotensored over simplicial sets, ordinary colimits automatically satisfy an enriched universal property and consequently define enriched functors. Hence, the only part of Garner's small object argument that fails to be enriched is what might be called ``step zero'': the functor \[ f \mapsto \coprod_i \coprod_{\mathrm{Sq}(i,f)} i. \]  This problem is resolved if we replace the coproduct over the set $\mathrm{Sq}(i,f)$ with a tensor with the appropriate enriched hom-space. The category of arrows in a bicomplete simplicial category $\aC$ is simplicially enriched with hom-spaces defined by pullback \[ \xymatrix{ \uSq (i,f) \ar[d] \ar[r] \ar@{}[dr]|(.2){\lrcorner} & \uC(\dom i, \dom f) \ar[d]^{f_*} \\ \uC(\cod i, \cod f) \ar[r]_{i^*} & \uC(\dom i, \cod f)}\] where we have written $\uC$ for the hom-space. A point in the simplicial set $\uSq(i,f)$ is precisely a commutative square from $i$ to $f$ in $\aC$. Simplicial tensors and cotensors in $\aC^{\mathbf{2}}$ are defined pointwise.

The step-one functorial factorization in the modified small object argument has the form:
\begin{equation}\label{eq:newstepone}\xymatrix{ \cdot \ar[d]_-{\coprod\limits_{i \in \aI} \uSq(i,f) \otimes i} \ar[r] \ar@{}[dr]|(.8){\ulcorner}& \dom f \ar@{=}[r] \ar[d]^-{C_1f} & \dom f \ar[d]^f \\ \cdot \ar[r] & E_1f \ar[r]_{F_1f} & \cod f}\end{equation} Compare with \eqref{eq:stepone}. Because tensors and colimits define enriched functors,  the step-one functorial factorization is simplicially enriched. The remainder of the construction proceeds as described in section~\ref{sec:background}. Because the (ordinary) colimits involved in the free monad construction satisfy a simplicially enriched universal property in any cotensored simplicial category, this construction produces simplicially enriched functors and natural transformations.

It remains to argue, for example in the case where $\aI$ is the set of generating cofibrations, that the left factor in the factorization  is a cofibration and the right factor is a trivial fibration. The proof uses the algebraic interpretation of Garner's construction mentioned in section~\ref{sec:background}. The right factor $F$ is  the algebraically free monad on the functor $F_1$, meaning that an arrow in $\aC$ admits an algebra structure for the monad $F$ if and only if it admits an algebra structure for the pointed endofunctor $F_1$. Unraveling the definition, this is the case if and only if there is a lift in the right-hand square of \eqref{eq:newstepone}, which  is the case if and only if there is a lift in the outside rectangle. 

For each $i \in \aI$, restricting to the appropriate vertex of the space $\uSq(i,f)$, a lift in the outer rectangle produces a solution to any lifting problem of $i$ against $f$. Hence, if $f$ is an algebra for the monad $F$ then it is a trivial fibration. The converse holds because $\uSq(i,f) \otimes -$ preserves cofibrations as a consequence of the SM7 axiom.  Thus, the trivial fibrations are precisely those maps admitting the structure of $F$-algebras. The map $Ff$ is a free $F$-algebra and hence a trivial fibration.

Finally, the arrow $Cf$ is a coalgebra for the comonad $C$ and this coalgebra structure can be used to solve any lifting problem against an $F$-algebra, i.e., against a trivial fibration. It follows that $Cf$ is a cofibration. 
\end{proof}

\begin{cor}\label{cor:simpresol} If $F \colon \aC \rightleftarrows \aD \colon G$ is a simplicial Quillen adjunction between cofibrantly generated simplicial model categories, then there exist derived resolutions as in Corollary \ref{cor:resolutions} comprised of simplicial functors and simplicial natural transformations.
\end{cor}

By the remark given at the end of section \ref{sec:htpyresol}, we can restrict these diagrams to the simplicial subcategories spanned by the fibrant-cofibrant objects. These subcategories are important, for instance to the construction of the associated $\infty$-categories, because their  hom-spaces are ``homotopically correct''.

\section{Goodwillie calculus in model categories and the chain
rule}\label{sec:calculus}

Applications of homotopical resolutions arise when studying Goodwillie
calculus for functors between homotopical categories \cite{goodwillie}.  Kuhn has
observed that the foundational definitions and theorems of Goodwillie
calculus for functors between based simplicial sets and spectra can be
carried out in any pointed simplicial model category $\aC$ that is
left proper and cellular or proper and combinatorial~\cite{kuhn}.
We will be interested in homotopical functors $F \colon \aC \to \aD$.
We say that $F$ is \emph{finitary} if it preserves filtered homotopy colimits
and \emph{reduced} if $F(*) \htp *$.  Given a homotopical functor
$F \colon \aC \to \aD$ one can define a Goodwillie tower
\[
\cdots \to P_{n+1}F \to P_n F \to P_{n-1} F \to \cdots
\]
where $P_n F$ is $n$-\emph{excisive}, i.e., sends strongly homotopy
cocartesian $(n+1)$-cubes to homotopy cartesian cubes.  We say that a
functor $F$ is \emph{homogeneous} of degree $n$ if it is $n$-excisive and
$P_{n-1}F \htp *$.  Moreover, one can define derivatives  
\[
D_n F = \hofib(P_n F \to P_{n-1} F)
\]
that are homogeneous of degree $n$.

This program has been carried out in detail in the forthcoming thesis
of Peirera~\cite{Peirera}.  We begin by reviewing the basic setup.
The key technical preliminary required by this generalization is that
from a left proper cellular pointed simplicial model category one can
construct a stable model category $\Stab(\aC)$ which again has these
properties \cite[5.7,A.9]{hovey}. A pointed
simplicial model category $\aC$ is equipped with a suspension functor
defined by pushout \[ \xymatrix{
X \otimes \partial \Delta^1 \ar[d] \ar[r]  \ar@{}[dr]|(.8){\ulcorner}
&X \otimes \Delta^1 \ar[d] \\ {*} \ar[r] & \Sigma X}\] As a
consequence of the SM7 axiom, suspension is a left Quillen
endofunctor.  

Furthermore, there is a Quillen adjunction 
\begin{equation}\label{eq:newloopssus}
\Sigma^\infty \colon \aC \rightleftarrows \Stab(\aC) \colon \Omega^\infty
\end{equation}
and moreover the category $\Stab(\aC)$ is entitled to be referred to
as ``the'' stabilization in the sense that any suitable functor from
$\aC$ to a stable category factors through $\Sigma^\infty$.  The
precise universal property satisfied by $\Stab(\aC)$ is easiest to
state in the setting of  $\infty$-categories; e.g.,
see~\cite[1.4.5.5]{HA}.

\begin{rem}
Although we work with the hypothesis of a left proper cellular pointed
model category for compatibility with~\cite{hovey}, one could
alternatively work with other similar hypotheses, e.g., proper
combinatorial model categories.
\end{rem}

One of the deeper and more surprising facts about Goodwillie calculus
is that $D_n F$, as an $n$-homogeneous functor, is always in the image
of $\Omega^\infty$.  Specifically, we have the following analogue
of~\cite[2.7]{AroneChing}.

\begin{prop}[Peirera]\label{prop:perdeloop}
Let $F \colon \aC \to \aD$ be a reduced simplicial finitary
homotopical functor.  Then there is an equivalence
\[
D_n F(X) \htp \Omega^\infty R (\mathbb{D}_n F) (\Sigma^{\infty} Q
X),
\]
where $\mathbb{D}_n F \colon \Stab(\aC) \to \Stab(\aD)$ is an
$n$-homogeneous homotopical functor.
\end{prop}

The fact that derivatives are stable objects tells us that the
relationship between $\aC$ and $\Stab(\aC)$ plays a basic role in the
structure of the Goodwillie calculus.  This insight has been used in
practical work on computing derivatives and Taylor towers.  Notably,
the celebrated work of Arone and Mahowald on the derivatives of the
identity functor from spaces to spaces depends on a cobar resolution
involving the $\Sigma^{\infty}\dashv \Omega^{\infty}$
adjunction~\cite{AroneMahowald}.  More recently, this technique has
been generalized to provide a key technical result Arone and Ching use
to deduce the chain rule for functors from spaces to
spaces from the chain rule for functors from
spectra to spectra.

In this setting, the source category of \eqref{eq:newloopssus},
simplicial sets, has all objects cofibrant in the standard model
structure and the target category, the EKMM category of
$S$-modules~\cite{EKMM}, has all objects fibrant.  Under these
hypotheses, the relevant cobar construction can be computed simply
from the ordinary cosimplicial resolution of the point set monad
$\Omega^\infty \Sigma^\infty$.  Of course in general, these
assumptions do not hold.  For instance, when $\aC$ is the category of
augmented commutative $R$-algebras and $\Stab(\aC)$ is the category of
$R$-modules~\cite{basterramandell}, not all objects in $\aC$ are
cofibrant.  But we use our homotopical resolutions to extend these
techniques to the more general setting of calculus in model
categories.  Specifically, we prove the following theorem, which is a
generalization of a key technical result of Arone and
Ching~\cite[16.1]{AroneChing} and~\cite[0.3]{AroneChing}.

\begin{thm}\label{thm:cobar}
Let $\aC$ be a pointed simplicial model category that is a left proper
and cellular.  Denote by $\Sigma^\infty$ and $\Omega^\infty$ the
adjoints connecting $\aC$ and $\Stab(\aC)$.  Let
$F \colon \aC \to \aC$ be a finitary pointed simplicial homotopical
functor and let $G \colon \aC \to \aC$ be a pointed
simplicial homotopical functor.  Then there are natural equivalences
\[
\xymatrix{
 \eta_n \colon P_n(FG) & \ar[l] \ar[r] P_n(FQG)
& \widetilde{\Tot}(P_n(FQ \Omega^\infty R (\Sigma^\infty
Q \Omega^\infty R)^{\bullet} \Sigma^\infty QG)) 
}
\]
and
\[
\xymatrix{
\epsilon_n\colon D_n(FG) & \ar[l] \ar[r] D_n(FQG)
& \widetilde{\Tot}(D_n(FQ \Omega^\infty R (\Sigma^\infty
Q \Omega^\infty R)^{\bullet} \Sigma^\infty QG)) 
}
\]
where $\widetilde{\Tot}$ denotes the $\Tot$ of a Reedy fibrant
replacement of the indicated cosimplicial objects, which are defined
analogously with definition~\ref{def:dercomp}.
(The result also holds for functors $F \colon \aC \to \Stab(\aC)$ and
$G \colon \Stab(\aC) \to \aC$.)
\end{thm}

In order to carry out the proof of this theorem, we need the
following technical proposition, which is a generalization
of~\cite[3.1]{AroneChing} and has the same proof.

\begin{prop}\label{prop:chingtech}
Let $\aC$ be a left proper cellular simplicial model category and let
$F$ and $G$ be pointed simplicial homotopical functors $\aC \to \aD$.
Assume that $F$ is finitary.  Then the natural map
\[
P_n(FG) \to P_n(F(P_n G))
\]
is an equivalence.
\end{prop}

We can now give the proof of the theorem.

\begin{proof}[Proof of Theorem~\ref{thm:cobar}]
The unit zig-zag 
\[
\xymatrix{
X & \ar[l]_q QX \ar[r]^-\iota & Q \Omega^\infty R \Sigma^\infty QX
}
\] 
gives rise to the rise to the zig-zags $\eta_n$ and $\epsilon_n$.
The induction in the proof of~\cite[16.1]{AroneChing} now goes
through essentially without change to establish the result for
$\eta_n$; none of the arguments depend on the details 
of the cobar construction except for the analysis of the case in which
$F = H \Sigma^{\infty}$.  In our context, we can instead assume that
$F = H R \Sigma^{\infty} Q$ by Proposition~\ref{prop:perdeloop}.
In this case, composing out the extra $Q$, the cosimplicial object
becomes
\[
\widetilde{\Tot}(P_n(H R \Sigma^{\infty} Q \Omega^\infty R
(\Sigma^\infty Q \Omega^\infty R)^{\bullet} \Sigma^\infty QG))
\]
and this has an extra codegeneracy induced by the ``counit''
composite
\[
\xymatrix{ R\Sigma^\infty Q\Omega^\infty R \ar[r]^-\pi & R}
\]
The case of $\epsilon_n$ now follows just as
in~\cite[16.1]{AroneChing}, using Proposition~\ref{prop:chingtech} in 
place of~\cite[3.1]{AroneChing}.
\end{proof}

We expect that Theorem~\ref{thm:cobar} should form the basis of a
proof of the Arone-Ching chain rule in the general setting of
Goodwillie calculus for an arbitrary model category.  However, we
note that there remain substantial technical issues involved in
closely mimicking their approach associated to the facts that:
\begin{enumerate}
\item Our functor $\Sigma^{\infty} Q \Omega^{\infty} R$ defines a
comonad on the homotopy category but not on the point set level
because the counit is only a zig-zag, and 
\item We do not know how to construct a symmetric monoidal fibrant
replacement functor on $\Stab(\aC)$.
\end{enumerate}
See the ``Technical remarks'' section of the preface
to~\cite{AroneChing} for further elaboration on the role these
hypotheses play in their work.

We also expect that Theorem~\ref{thm:cobar} should be useful in
proving the conjecture that the derivatives of the identity functor on
the category of $\aO$-algebras for a reasonable operad $\aO$
recovers the operad.

\section{2-categorical structure of homotopical resolutions}\label{sec:adjdata}

We conclude with a few observations about the 2-categorical structure
of the resolutions of Corollary \ref{cor:resolutions}.  We do not have
a specific application in mind.  However, certain aspects of homotopy
theory, sometimes in the guise of $\infty$-category theory, can be
productively guided by 2-category theory. One instance of this has to
do with the free \emph{homotopy coherent adjunction}, inspired by the
free 2-category $\Adj$ containing an adjunction, which satisfies an
analogous universal property having to do with adjunctions between
$\infty$-categories \cite{riehlverity}.   We believe that
interpreting homotopical resolutions in this context provides a
conceptual explanation for their efficacy.

The 2-category $\Adj$ has two objects $0,1$ and hom-categories \[ \Adj(0,0) = \Del_+ \quad \Adj(0,1) = \Del_{-\infty} \quad \Adj(1,0) = \Del_{\infty} \quad \Adj(1,1) = \Del_+^{\op}\] Here $\Del_+$ is the category of finite ordinals and order preserving maps and $\Del_{\infty}$ and its opposite $\Del_{-\infty}$ are the subcategories of non-empty ordinals and maps that preserve the top and bottom elements respectively. The composition of 1-cells with domain and codomain 0 is by ordinal sum \[ \Adj(0,0) \times \Adj(0,0) = \Del_+ \times \Del_+ \xrightarrow{\oplus} \Del_+ = \Adj(0,0).\] This composition extends to 2-cells by ``horizontal juxtaposition'' of order preserving maps.  Composition for endo-1- and 2-cells of 1 is defined analogously, and these composition operations restrict to the subcategories $\Del_{-\infty}$ and $\Del_\infty$ defining the other hom-categories. See \cite{schanuelstreet} or \cite{riehlverity} for more details.

Suppose given a (simplicial) Quillen adjunction $F \colon \aC \rightleftarrows \aD \colon G$ between cofibrantly generated (simplicial) model categories. By Corollaries \ref{cor:resolutions} and \ref{cor:simpresol} there exist diagrams \eqref{eq:cosimpresol} of (simplicial) functors and (simplicial) natural transformations between the (fibrant simplicial) categories $\aC^{\cf}$ and $\aD^{\cf}$. This data assembles into a morphism $H \colon \Adj \to \hCat$, where the target is the 2-category of (simplicial) homotopical categories, (simplicial) homotopical functors, and (simplicial) natural weak equivalences---but this map is not a strict 2-functor. Instead it is a coherent mix of a lax and colax functor in a way we shall now describe.

The two-object 2-category $\Adj$ has two full sub 2-categories, which we might denote $B\Del_+$ and $B\Del_+^{\op}$; here the ``$B$'' denotes a ``delooping'' of the monoidal category appearing as its hom-category. The map $H$ restricts to a colax functor $B\Del_+$ and a lax functor $B\Del_+^{\op}$ meaning there are diagrams
\[ \xymatrix{ B\Del_+ \times B\Del_+ \ar[d]_{QGRFQ^\bullet \times QGRFQ^\bullet} \ar[r]^-\oplus \ar@{}[dr]|{\delta\Downarrow} & B\Del_+ \ar[d]^{QGRFQ^\bullet} & & \mathbf{1} \ar[r]^-\id \ar[dr]_-\id^{\labelstyle q\Downarrow} & B\Del_+ \ar[d]^{QGRFQ^\bullet}  \\ \hCat \times \hCat \ar[r]_-\times & \hCat & & & \hCat}\] 
\[ \xymatrix{ B\Del_+^{\op} \times B\Del_+^{\op} \ar[d]_{RFQGR_\bullet \times RFQGR_\bullet} \ar[r]^-\oplus \ar@{}[dr]|{\mu\Uparrow} & B\Del_+^{\op} \ar[d]^{RFQGR_\bullet} & & \mathbf{1} \ar[r]^-\id \ar[dr]_-\id^{\labelstyle r\Uparrow} & B\Del_+^{\op} \ar[d]^{RFQGR_\bullet}  \\ \hCat \times \hCat \ar[r]_-\times & \hCat & & & \hCat}\] 
whose 2-cells satisfy associativity and unit coherence conditions. Here the monoidal product $\oplus$ on $\Del_+$ defines the composition functor for the 2-category $B\Del_+$.  When $H$ is extended back to the entirety of $\Adj$, there exist six similar ``composition'' diagrams in which the direction of the 2-cell is determined by whether the composition takes place at the object $0$ or at the object $1$. The displayed 2-cells again satisfy the usual (op)lax coherence conditions, up to the fact that some point in the wrong direction. 

One way to interpret this data is that if $H$ is composed with the
2-functor $\Ho \colon \hCat \to \Cat$ that sends a homotopical
category to its homotopy category, the result is a pseudofunctor
$\Adj \to \Cat$ which classifies the total derived adjunction of the
Quillen adjunction $F \dashv G$.

\begin{appendix}

\section{Producing simplicial enrichments}\label{appendix:simp}

A generic bicomplete category need not be simplicially enriched,
tensored, and cotensored; a model category with these properties need
not be a simplicial model category. However, under certain set
theoretical hypotheses a model category may be replaced by a Quillen
equivalent simplicial model category.  Moreover, this replacement is
functorial in the sense that Quillen adjunctions can be lifted to
simplicial Quillen adjunctions.  We review the results of
Dugger~\cite{dugger} and Rezk-Schwede-Shipley~\cite{RSS} that describe
this replacement.

\begin{thm}[Dugger]\label{thm:dugger} Let $\aC$ be a left proper combinatorial model category. Then there exists a simplicial model structure on the category $s\aC$ of simplicial objects in $\aC$ whose cofibrations are Reedy cofibrations and whose fibrant objects are Reedy fibrant simplicial objects whose structure maps are  weak equivalences. Furthermore, the adjunction $c \colon \aC \rightleftarrows s\aC \colon \mathrm{ev}_0$ is a Quillen equivalence.
\end{thm}

By a result of Joyal, a model structure is completely determined by the cofibrations and the fibrant objects, supposing it exists. The fibrant objects are called \emph{simplicial resolutions} in \cite{dugger} and \emph{homotopically constant} simplicial objects in \cite{RSS}.

\begin{proof}[Proof sketch]
For any category $\aC$ admitting certain limits and colimits, $s\aC$ is tensored, cotensored, and enriched in a standard way: the tensor of a simplicial set $K$ with a simplicial object $X$ is defined by $(K \otimes X)_n = \coprod_{K_n} X_n$. The cotensor and enrichment are determined by adjunction. Adopting the terminology of \cite{RSS}, the simplicial category $s\aC$ admits a \emph{canonical simplicial model structure} defined to be a particular left Bousfield localization of the Reedy model structure on $s\aC$; in particular, its cofibrations are the Reedy cofibrations. Weak equivalences are those maps of simplicial objects that induce a weak equivalence on (corrected) homotopy colimits. Note, in particular, that pointwise weak equivalences have this property by tautology.
\end{proof}

Work of Rezk, Schwede, and Shipley extends these results to Quillen adjunctions. This story is somewhat subtle. Dugger's result uses the machinery of left Bousfield localization for which no general characterization of the resulting fibrations exist. This makes it difficult to detect right Quillen functors. Here, a general model categorical result will prove useful. 

\begin{lem}[{\cite[A.2]{dugger}}]\label{lem:RQuillen} An adjoint pair of functors between model categories is a Quillen adjunction if and only if the right adjoint preserves trivial fibrations and fibrations between fibrant objects.
\end{lem}

Rezk-Schwede-Shipley observe that the fibrations between fibrant objects in the canonical model structure are precisely the Reedy fibrations \cite[3.9]{RSS}. This enables the proof of the result we want.

\begin{prop}[{\cite[6.1]{RSS}}]\label{prop:RSS} Suppose $F \colon \aC \rightleftarrows \aD \colon G$ is a Quillen adjunction between left proper combinatorial model categories. Then $\overline{F} \colon s\aC \rightleftarrows s\aD \colon \overline{G}$ is a simplicially enriched Quillen adjunction between the canonical simplicial model categories.
\end{prop}
\begin{proof}
An adjoint pair of functors $F \colon \aC \rightleftarrows \aD \colon G$ induces an adjoint pair $\overline{F} \colon s\aC \rightleftarrows  s\aD \colon \overline{G}$, defined pointwise. This latter adjunction is simplicially enriched: by a general categorical result an a priori unenriched adjoint pair of functors between tensored simplicial categories admits the structure of a simplicial adjunction if and only if the left adjoint preserves simplicial tensors up to coherent natural isomorphism. Here \[ F(K \otimes X)_n = F ((K \otimes X)_n) = F( \coprod_{K_n} X_n) \cong \coprod_{K_n} F X_n = K \otimes FX\] is the desired isomorphism.

Furthermore, if the original $F$ and $G$ are Quillen, then the prolonged adjunction is Quillen with respect to the Reedy model structures: $G$ preserves the limits defining mapping objects and (trivial) fibrations in $\aD$ characterizing Reedy (trivial) fibrations in $s\aD$. Hence, $G$ preserves Reedy (trivial) fibrations and Reedy fibrant objects. Fibrant objects in the simplicial model structure on $s\aD$ are Reedy fibrant simplicial objects whose structures maps are weak equivalences. The objects in Reedy fibrant simplicial objects are fibrant in $s\aD$; hence $G$ preserves fibrant objects. Now the result follows immediately from Lemma \ref{lem:RQuillen} and the characterization of fibrations between fibrant objects in the canonical simplicial model structure.
\end{proof}

It follows from Theorem \ref{thm:dugger},  Proposition \ref{prop:RSS}, and the commutative diagram \[  \xymatrix{ \aC \ar@<1ex>[r]^F \ar@{}[r]|{\perp} \ar@<-1ex>[d]_c \ar@{}[d]|{\dashv} & \aD \ar@<1ex>[l]^G \ar@<-1ex>[d]_c \ar@{}[d]|{\dashv} \\ s\aC \ar@<1ex>[r]^{\overline{F}} \ar@{}[r]|{\perp} \ar@<-1ex>[u]_{\mathrm{ev}_0} & s\aD \ar@<-1ex>[u]_{\mathrm{ev}_0} \ar@<1ex>[l]^{\overline{G}}}\] that any Quillen adjunction between left proper combinatorial model categories can be replaced by a Quillen equivalent simplicial Quillen adjunction. 

If $\aC$ permits the small object argument, so does $s\aC$. 

\begin{cor} For any Quillen adjunction between left proper combinatorial model categories there exists a Quillen equivalent Quillen adjunction for which one may construct simplicially enriched homotopical resolutions as described in Corollary \ref{cor:resolutions}.
\end{cor}

\end{appendix}

\end{document}